\documentclass[12pt]{amsart}
\usepackage{graphicx}
\usepackage{tikz}
\usepackage[dvipsnames]{xcolor}
\usepackage{amsmath}
\usepackage{caption}
\usepackage[top=1.5in, bottom=1.5in, left=1.5in, right=1.5in]{geometry}
\usepackage{float} 
\usepackage{mathrsfs}
\theoremstyle{plain}
\newtheorem{theorem}{Theorem}[section]
\newtheorem{corollary} [theorem]{Corollary}
\newtheorem{lemma}[theorem]{Lemma}
\theoremstyle{definition}%

\usepackage{fancyhdr}
\usepackage{amsmath,amssymb,mathrsfs}
\begin{document}

\title[Realization and Classification of Hamiltonian-Circle
Multisigns]{Realization and Classification of Hamiltonian-Circle
Multisigns}\footnote{}

\author{Xiyong Yan}
\address{89 Park Ave,  Binghamton, NY,  USA,  13903.}
\email{xiyongyan@gmail.com}
\begin{abstract}
We investigate the multisigns of Hamiltonian circles in the multisigned complete graph 
\(\Sigma_n := (K_n, \sigma, \mathbb{F}_2^m)\).
The \emph{multisign} of a circle \(C\) is defined as the sum
\[
\sigma(C) := \sum_{e \in E(C)} \sigma(e).
\]
For a fixed \(m\) and sufficiently large \(n\), we show that the set of multisigns of Hamiltonian circles
\[
\{\sigma(H) : H  \text{ is a Hamiltonian circle of }  \Sigma_n)\}
\]
forms either a subspace, an affine subspace, or the entire space \(\mathbb{F}_2^m\),
except in certain exceptional cases.
\end{abstract}

\maketitle
\section{Introduction}

Our research question originates from the problem collection 
\emph{``Negative (and Positive) Circles in Signed Graphs''}~\cite{1}, 
where Professor Zaslavsky asks: 
\emph{If a signed graph is unbalanced and has a Hamiltonian circle, 
must it contain a negative Hamiltonian circle? A positive one?} 
In the paper "Hamiltonian cycles in signed and multisigned complete graphs"~\cite{2}, we proved that a signed complete graph on~$n$ vertices 
contains both a positive and a negative Hamiltonian circle 
if and only if it also contains both a positive and a negative triangle. This result naturally motivates the extension from single to multiple sign systems, 
where each edge carries more than one independent sign. 

In the paper \emph{``Double Signs of Hamiltonian Circles in Doubly Signed Complete Graphs''}~\cite{3}, 
we established conditions under which Hamiltonian circles realize 
all four possible double signs, and proved that this occurs whenever 
the set of triangle double signs contains at least three distinct values.

In this paper, we consider the case of \emph{multisigned complete graphs}, 
where each edge of the complete graph~$K_n$ is labeled by an element of~$\mathbb{F}_2^m$. 
For a circle~$C \subseteq K_n$, its \emph{multisign} is defined as the sum of the labels of its edges. 
We establish conditions under which Hamiltonian circles realize all possible multisigns. 
However, in certain special cases, not all multisigns of Hamiltonian circles can be realized. 
Several examples illustrating these exceptional cases will be presented later.

\section{Preliminaries}

\noindent\textbf{Definition (Multisigned complete graph).}
A \emph{multisigned complete graph} is denoted  $\Sigma_n:=(K_n,\sigma,\mathbb{F}_2^m)$ in which $K_n=(V,E)$ is the complete graph on $n$ vertices and $\sigma:E\to\mathbb{F}_2^m$ assigns to each edge an $m$-bit vector (a “multisign”).

\noindent\textbf{Definition (Hamiltonian Circle).}
 A \emph{Hamiltonian circle} in $\Sigma_n$ is a subgraph $C \subseteq G$ such that $C$ is a simple cycle containing every vertex of $\Sigma_n$ exactly once.

\noindent\textbf{Definition  (Covering \(C_4\)-necklace).} 
Let \(D_1, D_2, \dots, D_t\) be 4-circles (\(C_4\)'s) in \(\Sigma_n\), 
and let \(p_1, p_2, \dots, p_\ell\) be paths in \(\Sigma_n\).
A \emph{covering \(C_4\)-necklace} is a subgraph of \(\Sigma_n\) consisting of 
the 4-circles \(D_1, D_2, \dots, D_t\) and the paths \(p_1, p_2, \dots, p_\ell\) satisfying the following conditions:
\begin{enumerate}
    \item For each \(i = 1, 2, \dots, t\), let $v_{i1}$ and $v_{i3}$ be two opposite vertices of \(D_i\). 
Each of $v_{i1}$ and $v_{i3}$ is only connected to an endpoint of a path \(p_i\) 
or  a vertex of \(D_j\) for some \(j \neq i\) .

    \item the union of all \(D_i\)'s and \(p_i\)'s forms a closed chain. See Figure \ref{fig:diamond-necklace}.
    \item The collection \(\{D_1, D_2, \dots, D_t, p_1, p_2, \dots, p_\ell\}\) covers all vertices of $\Sigma_n$.
\end{enumerate}

\noindent\textbf{Definition (Normalization of a Vertex).}
Let $\Sigma_n = (K_n, \sigma, \mathbb{F}_2^m)$ be a multisigned complete graph.  
To \emph{normalize} a vertex $v \in V(K_n)$ means to apply a switching function 
on all vertices other than $v$ so that, after switching,  
every edge incident to $v$ has multisign $e:=(0,0,...,0)$.  

\section{
Realization and Classification of Hamiltonian-Circle Multisigns}

\begin{lemma} [$C_4$-necklace Lemma]
Suppose that \(\Sigma_n:=(K_n,\sigma,\mathbb{F}_2^m)\) contains a covering diamond necklace consisting of paths \(p_1, p_2, \dots, p_l\) and squares (or $C_4$'s)
\[
D_1, D_2, \dots, D_t,
\]
and that the set
\[
\{\sigma(D_1), \sigma(D_2), \dots, \sigma(D_t)\}
\]
 spans the vector space \(\mathbb{F}_2^m\). Then \(\Sigma_n\) exhibits all possible multisigns of Hamiltonian circles.

\end{lemma}
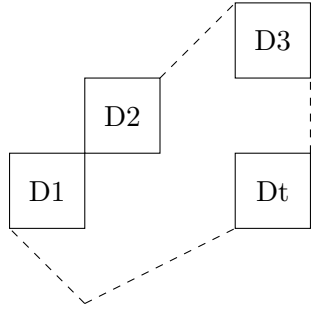
\begin{figure}[H]
    \centering
 
\begin{tikzpicture}[font=\small, scale=1.0]

  \draw (0,0) rectangle (1,1);
  \node at (0.5,0.5) {D1};

  \draw (1,1) rectangle (2,2);
  \node at (1.5,1.5) {D2};

  \draw (3,0) rectangle (4,1);
  \node at (3.5,0.5) {Dt};

  \draw (3,2) rectangle (4,3);
  \node at (3.5,2.5) {D3};

  \draw[dashed] (2,2) -- (3,3);


  \draw[dashed] (1,-1) -- (3,0);
 \draw[dashed] (1,-1) -- (0,0);
  \draw[dashed] (4,1) -- (4,2);

\end{tikzpicture}
\caption{A $C_4$-necklace with  squares \(D_1, D_2,..., D_t\) when $t=4$.}
\label{fig:diamond-necklace}
\end{figure}

\begin{proof}
For \(i = 1,2,\dots,t\), construct the \(i\)th square \(D_i\). Label two  vertices of $D_i$ that connect to  $D_j$ (for $j\neq i$) or a path as \(v_{i1}\) and \(v_{i3}\), and label the remaining two as \(v_{i2}\) and \(v_{i4}\). Add an edge between \(v_{i2}\) and \(v_{i4}\); this transforms the square \(D_i\) into a diamoned \(D_i'\). In \(D_i'\), there are two Hamiltonian paths \(p_{0}^i\) and \(p_{1}^i\) that start at \(v_{i1}\) and end at \(v_{i3}\) with multisigns \(\sigma(p_{0}^i)=x_i\) and \(\sigma(p_{1}^i)=\sigma(D_i) + x_i\), for some $x_i\in \mathbb{F}_2^m$.

Let \(g \in  \mathbb{F}_2^m\) be arbitrary. Our goal is to show that there exists a Hamiltonian circle $H$ (in \(\Sigma_n\)) such that 
$\sigma(H)=g.$

Since the paths \(p_1, p_2, \dots, p_l\) are fixed, each of them will have fixed multisign.   Let 
\[
h := \sum_{k=1}^l \sigma(p_k).
\]

Observe that \(g\), \(h\), and each \(x_i\) belong to \( \mathbb{F}_2^m\). Thus, the element
\[
g - h - \sum_{i=1}^t x_i
\]
lies in \( \mathbb{F}_2^m\). Since \(\{\sigma(D_1), \sigma(D_2), \dots, \sigma(D_t)\}\) spans $\mathbb{F}_2^m$, there exists a sequence \(a_1, a_2, \dots, a_t \in \{0,1\}\) such that
\[
\sum_{i=1}^t a_i \sigma(D_i) = g - h - \sum_{i=1}^t x_i.
\]

Set
\[
\sigma(p_{a_i}^i) = x_i + a_i \sigma(D_i).
\]
 Therefore,
\[
\sum_{i=1}^t \sigma(p_{a_i}^i) = \sum_{i=1}^t x_i + \sum_{i=1}^t a_i \sigma(D_i).
\]

It follows that
\[
\sum_{i=1}^t \sigma(p_{a_i}^i) = \sum_{i=1}^t x_i + \sum_{i=1}^t a_i \sigma(D_i) = \sum_{i=1}^t x_i + \Bigl(g - h - \sum_{i=1}^t x_i\Bigr) = g - h.
\]
That is $$g=\sum_{i=1}^t \sigma(p_{a_i}^i)+h=\sum_{i=1}^t \sigma(p_{a_i}^i)+ \sum_{k=1}^l \sigma(p_k).$$ Observe that the paths  $p_{a_i}^i$'s and $p_k$'s form a Hamiltonian circle in $\Sigma_n$ and this  Hamiltonian circle has multisign $g$. 
This completes the proof.
\end{proof}

\vspace{1cm}

The $C_4$-necklace Lemma is the main tool used to prove our theorem. 
The difficult part of the theorem is to show the existence of a $C_4$-necklace. 
Before we proceed to the next section, we present a preliminary fact. 

Let \(\mathcal{Z}(\Sigma_n)\) denote the cycle space of the underlying graph of \(\Sigma_n\), 
and let \(\sigma(\mathcal{Z}(\Sigma_n))\) be the set of multisigns of all cycles in \(\Sigma_n\). 
Notice that \(\sigma(\mathcal{Z}(\Sigma_n))\) is  a subspace of~$\mathbb{F}_2^m$. 
Suppose \(\Sigma_n := (K_n, \sigma, \mathbb{F}_2^m)\), 
and if \(\sigma(\mathcal{Z}(\Sigma_n))\) is a proper subspace~$U$ of~$\mathbb{F}_2^m$, 
we may redefine \(\Sigma_n := (K_n, \sigma, U)\) 
so that it satisfies the condition of the theorem, 
and the theorem can then be applied.

\noindent\textbf{Definition (almost-disjoint).} (See Figure~\ref{dis}) 
We call triangles \(T_1, T_2, \dots, T_r\) \emph{almost-disjoint} if the following conditions hold:
\begin{enumerate}
\item Any two distinct triangles share at most one common vertex. 
\item Among any three triangles, at least one is disjoint from the other two.
\end{enumerate}

\begin{figure}[H]
    \centering
\begin{tikzpicture}[scale=1.2,
  dot/.style={circle,fill=black,inner sep=1.2pt},
  every node/.style={font=\small}
]

\node[dot,label=below:$v_1$] (v1) at (0,0) {};
\node[dot,label=below:$v_2$] (v2) at (2,0) {};
\node[dot,label=below:$v_3$] (v3) at (3.5,0) {};
\node[dot,label=below:$v_4$] (v4) at (6,0) {};
\node[dot,label=below:$v_5$] (v5) at (8,0) {};
\node[dot,label=below:$v_{10}$] (v10) at (4,0) {};

\draw (v1)--(v2)--(v3);
\draw  (v4)--(v5);

\node[dot,label=above:$v_9$] (v9) at (1.2,1.2) {};
\node[dot,label=above:$v_8$] (v8) at (3.2,1.8) {};
\node[dot,label=above:$v_7$] (v7) at (5.0,1.6) {};
\node[dot,label=above:$v_6$] (v6) at (8.0,1.2) {};

\node[dot,label=above:$v_{11}$] (v11) at (8.5,0) {};
\node[dot,label=above:$v_{12}$] (v12) at (9.5,1) {};
\node[dot,label=above:$v_{13}$] (v13) at (8.5,1.2) {};


\draw (v11)--(v13);
\draw (v11)--(v12)--(v13)--cycle;
\draw (v1)--(v2)--(v9)--cycle; 
\draw (v2)--(v3)--(v8)--cycle; 
\draw (v10)--(v4)--(v7)--cycle; 

\draw (v4)--(v5)--(v6)--cycle; 
\draw (v1)--(v9);
\draw (v2)--(v8);
\draw (v10)--(v7);
\draw (v4)--(v6);
\end{tikzpicture}
\caption{almost-disjoint triangles.}
\label{dis}
\end{figure}
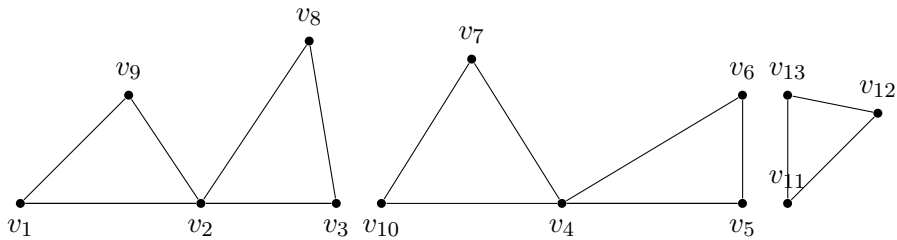

\noindent\textbf{Remark.} 
Later on, if we replace triangles by $C_4$'s and say they are almost-disjoint, 
the same conditions apply. 
The definition also applies when triangles are replaced by triangular paths.

In our main theorem, we exclude exceptional configurations and focus on the general (or normal) case. 
Given $\Sigma_n$, for sufficiently large~$n$, 
if the multisigns of triangles span the space $\mathbb{F}_2^m$ for some~$m$, 
then there likely exist $m$ almost-disjoint triangles with multisigns 
$x_1, x_2, \ldots, x_m$ such that 
\[
\mathbb{F}_2^m = \operatorname{span}\{x_1, x_2, \ldots, x_m\}.
\]
However, if such almost-disjoint triangles do not exist, there are cases in which not all multisigns of Hamiltonian circles can be realized.
We will present special cases in which $\Sigma_n$ fails to realize all 
multisigns of Hamiltonian circles, as well as cases in which all multisigns 
can indeed be realized.

\textbf{Example.} 
Suppose in $\Sigma_n := (K_n, \sigma, \mathbb{F}_2^m)$ we normalize a vertex~$v$. 
Let $e_1, e_2, \ldots, e_m$ be edges in $\Sigma_n \setminus v$ such that 
$\sigma(e_1), \sigma(e_2), \ldots, \sigma(e_m)$ are linearly independent. 
All other edges in $\Sigma_n \setminus v$ have the same multisign~$x$. 
Under this condition, we consider two cases to analyze whether all multisigns of Hamiltonian circles can be realized:

\textbf{Case 1 (Not all multisigns realized).} 
If any three or more edges among $e_1, e_2, \ldots, e_m$ share a common vertex or form a circle, 
then $\Sigma_n$ cannot realize all multisigns of Hamiltonian circles. 
For instance, if $x = e$, $m=3$ and $e_1, e_2, e_3$ share a common vertex  or form a circle, 
then no Hamiltonian circle will have the multisign 
$\sigma(e_1) + \sigma(e_2) + \sigma(e_3)$.

\textbf{Case 2 (All multisigns still realized).} 
If the edges $e_1, e_2, \ldots, e_m$ lie on a path, this may violate the almost-disjoint condition; 
however, in this case, we can still realize all multisigns of Hamiltonian circles. For instance, let 
\(\Sigma_n = (K_n, \sigma, \mathbb{F}_2^4)\) 
(see Figure~\ref{triangles}). 
Suppose that 
$\sigma(v_1v_2), \sigma(v_2v_3), \sigma(v_3v_4), \sigma(v_4v_5) \in \mathbb{F}_2^4$ 
are linearly independent, and that all other edges have the same multisign~$x$. 
Without loss of generality, let $x = e$. 
Suppose we wish to find a Hamiltonian circle with multisign 
$\sigma(v_1v_2) + \sigma(v_2v_3) + \sigma(v_4v_5)$. 
We can take the path 
$p := v_1v_2v_3v_7v_4v_5$, 
and connect the remaining vertices to form another path~$p_2$. 
Finally, we connect one endpoint of~$p$ to one endpoint of~$p_2$, 
and the other endpoint of~$p$ to the remaining endpoint of~$p_2$. 
The resulting graph is a Hamiltonian circle with multisign 
$\sigma(v_1v_2) + \sigma(v_2v_3) + \sigma(v_4v_5)$. In this way, for any element $y \in \mathbb{F}_2^4$, we can always find a Hamiltonian circle whose multisign is $y$.
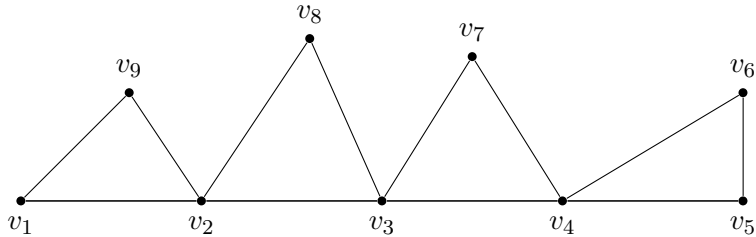
\begin{figure}[H]
    \centering
\begin{tikzpicture}[scale=1.2,
  dot/.style={circle,fill=black,inner sep=1.2pt},
  every node/.style={font=\small}
]

\node[dot,label=below:$v_1$] (v1) at (0,0) {};
\node[dot,label=below:$v_2$] (v2) at (2,0) {};
\node[dot,label=below:$v_3$] (v3) at (4,0) {};
\node[dot,label=below:$v_4$] (v4) at (6,0) {};
\node[dot,label=below:$v_5$] (v5) at (8,0) {};

\draw (v1)--(v2)--(v3)--(v4)--(v5);

\node[dot,label=above:$v_9$] (v9) at (1.2,1.2) {};
\node[dot,label=above:$v_8$] (v8) at (3.2,1.8) {};
\node[dot,label=above:$v_7$] (v7) at (5.0,1.6) {};
\node[dot,label=above:$v_6$] (v6) at (8.0,1.2) {}; 

\draw (v1)--(v2)--(v9)--cycle; 
\draw (v2)--(v3)--(v8)--cycle; 
\draw (v3)--(v4)--(v7)--cycle; 
\draw (v4)--(v5)--(v6)--cycle; 
\draw (v1)--(v9);
\draw (v2)--(v8);
\draw (v3)--(v7);
\draw (v4)--(v6);
\end{tikzpicture}
\caption{Four triangles are not almost-disjoint.}
\label{triangles}
\end{figure}

\noindent\textbf{Definition (Consecutive triangles).} 
We call two triangles \(T_i\) and \(T_j\)  
\emph{consecutive} if they share a common edge.

\noindent\textbf{Definition (Triangular Path).} (See Figure \ref{fig:placeholder})
Let $\Sigma_n = (K_n, \sigma, \mathbb{F}_2^m)$ be a multisigned complete graph.  
For each index $i$, the $i$-th \emph{triangular path} $P_i$ is the subgraph of $K_n$ formed by a sequence of consecutive triangles
\[
T_i,\, Y_{i1},\, Y_{i2},\, \dots,\, Y_{i r_i},\, R_i,
\] where $r_i+2$ is the length of $P_i$,
such that:
\begin{enumerate}
    \item Each pair of consecutive triangles in the sequence share a common edge.
    \item Nonconsecutive triangles are edge-disjoint.
    \item The first and last triangles, \(T_i\) and \(R_i\), are called the \emph{end triangles} of \(P_i\). \(T_i\) and \(R_i\) share an edge if and only if $r_i=0.$
\end{enumerate}

Let $\mathcal{H}(\Sigma_n)$ be the set of all Hamiltonian circles in $\Sigma_n.$

Let $\mathcal{S}(\Sigma_n)$ be The set of all multisigns of Hamiltonian circles in $\Sigma_n.$

Let $\eta(x)$ denote the number of triangles with multisign $x$. 

\begin{theorem}\label{lemma4}
Let 
\(\Sigma_n := (K_n, \sigma, \mathbb{F}_2^m)\)
be a multisigned complete graph. 
Suppose 
$\eta(e)\geq \eta(t)$ for all $t\in \mathbb{F}_2^m$ and  there exist triangles whose multisigns
\(x_1, x_2, \dots, x_m \in \mathbb{F}_2^m\) and 
are linearly independent. 
Excluding special cases, for sufficiently large \(n\), 
the multisigned complete graph \(\Sigma_n\) realizes all multisigns of Hamiltonian circles.
\end{theorem}

\begin{proof}
Excluding the special cases, we prove the theorem by considering two cases: whether or not there exist almost-disjoint triangles \(T_1, T_2, \dots, T_m\)  whose multisigns
\(\sigma(T_1), \sigma(T_2), \dots, \sigma(T_m)\) 
are linearly independent. In each case, we will show that $\Sigma_n$ realizes all multisigns
of Hamiltonian circles.

\textbf{Case 1.} Suppose there exist almost-disjoint triangles 
\(T_1, T_2, \dots, T_m\) such that
\[
\sigma(T_1), \sigma(T_2), \dots, \sigma(T_m)
\]
are linearly independent. Suppose there exist disjoint triangles 
\(R_1, R_2, \dots, R_m\) 
with \(\sigma(R_i) = e\), such that they are also disjoint from each triangle $T_i$. 

Choose triangular paths 
\(P_1, P_2, \dots, P_m\)
such that each path \(P_i\) connects \(T_i\) to \(R_i\).
Moreover, each pair of paths \(P_i\) and \(P_j\) is either disjoint or shares exactly one common vertex \(v \in V(T_i) \cap V(T_j)\).

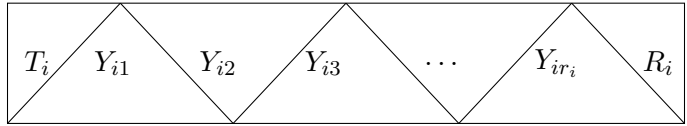
\begin{figure}[H]
    \centering

\begin{tikzpicture}[scale=1, every node/.style={font=\small}]

\begin{scope}[yshift=-2.5cm]
\draw (0,0) rectangle (9,1.6);

\draw (0,0)--(1.5,1.6);
\draw (1.5,1.6)--(3,0);
\draw (3,0)--(4.5,1.6);
\draw (4.5,1.6)--(6,0);
\draw (6,0)--(7.5,1.6);
\draw (7.5,1.6)--(9,0);

\node at (.4,0.8) {$T_i$};
\node at (1.4,0.8) {$Y_{i1}$};
\node at (2.8,0.8) {$Y_{i2}$};
\node at (4.2,0.8) {$Y_{i3}$};
\node at (5.8,0.8) {$\dots$};
\node at (7.3,0.8) {$Y_{ir_i}$};
\node[right] at (8.3,0.8) {$R_i$};
\end{scope}


\end{tikzpicture}
 
    \caption{The triangular path $P_i$}
    \label{fig:placeholder}
\end{figure}
\smallskip
\noindent

For each \(i =1,2,...,m\), we consider the triangular path \(P_i\). Within \(P_i\),  we aim to find a quadrilateral 
\(Q_i\), 
formed by two adjacent triangles, so that after all $Q_i$ are constructed, the vectors
\(\sigma(Q_1),\sigma(Q_2),...,\sigma(Q_m)\) 
are linearly independent. 

Let $W:=\operatorname{span}\{\sigma(T_1),\sigma(T_2),..., \sigma(T_m)\}.$

Let $$U_1:=\operatorname{span}\{ \sigma(T_2),\sigma(T_3),...,\sigma(T_m) \}.$$

Consider the triangular path $P_1$; see Figure~ \ref{fig:placeholder} with $i=1.$  If $\sigma(Y_{11}) \in U_{1},$ then  \(\sigma(T_1)+\sigma(Y_{11}) \in \sigma(T_1)+ U_{1}\). (Note: Here $\sigma(Y_{11})$ means $\sigma(Y_{1,1}).$ )
We choose the triangles \(T_1\) and \(Y_{11}\),
and define
\[
Q_1 = T_1 \,\Delta\, Y_{11},
\qquad 
\sigma(Q_1) = \sigma(T_1 \,\Delta\, Y_{11})
= \sigma(T_1) + \sigma(Y_{11}).
\]

We may write
\[
\sigma(Q_1)
= \sigma(T_1) +\sum_{j =2}^m a_{1j} \, \sigma(T_j), \text{ for some}
\quad a_{1j} \in \{0,1\}.
\]

\smallskip
If instead \(\sigma(Y_{11}) \notin U_{1}\), then \(\sigma(Y_{11}) \in U_{1}+\sigma(T_1)\). Indeed, since the vectors $\sigma(T_1),\dots,\sigma(T_m)$ are linearly independent, we have $\sigma(T_1)\notin U_1$.
Hence a disjoint union of two cosets $ U_1 \,\sqcup\, (\sigma(T_1)+U_1)=W$.
 Since every triangle in \(P_1\) lies in \(\Sigma_n\), each multisign \(\sigma(Y_{1j})\) belongs to \(W\). Thus, since $\sigma(Y_{11})\in W$ and it does not lie in $U_1$,  it must lie in the other coset $\sigma(T_1)+U_1$.
We then examine the next triangle \(Y_{12}\).
If \(\sigma(Y_{12}) \in U_{1}\), 
we choose the triangles \(Y_{12}\) and \(Y_{11}\), and
we set
\[
Q_1 = Y_{11} \,\Delta\, Y_{12}.
\]
Otherwise, if  \(\sigma(Y_{12}) \notin U_1,\) then  \(\sigma(Y_{12}) \in \sigma(T_1)+ U_1\). 
We continue to the next triangle \(Y_{13}\), and so on.
We proceed in this manner until we find two adjacent triangles 
whose multisigns have the property that
one lies in \(U_1+\sigma(T_1)\) while the other does not.
We then define \(Q_1\) as their symmetric difference.

\smallskip
If all of the multisigns 
\(\sigma(Y_{11}), \sigma(Y_{12}), \dots, \sigma(Y_{1r_1})\)
belong to \(U_1+\sigma(T_1)\),
we instead select \(Y_{1r_1}\) and \(R_1\), and define
\[
Q_1 = Y_{1r_1} \,\Delta\, R_1.
\]
Then
\[
\sigma(Q_1)
= \sigma(Y_{1r_1}) + \sigma(R_1)
= \sigma(Y_{1r_1}) + e
= \sigma(T_1) + \sum_{j =2}^m a_{1j} \, \sigma(T_j), 
\quad a_{1j} \in \{0,1\}.
\]

\smallskip
Therefore, in all cases we can express
\[
\sigma(Q_1)
=\sigma(T_1) + \sum_{j =2}^m a_{1j} \, \sigma(T_j).
\]

Then, replace $\sigma(T_1)$ in $\{ \sigma(T_1),\sigma(T_2),...,\sigma(T_m) \}$ with $\sigma(Q_1).$

Now we show   that 
\[\mathcal{T}_1:=\{\sigma(Q_1),\sigma(T_2),...,\sigma(T_m)\}\] 
is linearly independent. Indeed, since $\{\sigma(T_2),...,\sigma(T_m)\}$ is linearly independent, and by our construction,  \[
\sigma(Q_1)
=\sigma(T_1) + \sum_{j =2}^m a_{1j} \, \sigma(T_j).
\] 
Since $\sum_{j =2}^m a_{1j} \, \sigma(T_j)\in U_1$ and $\sigma(T_1)\notin U_1 $, it follows that $\sigma(Q_1)\notin U_1 $. Hence, $\mathcal{T}_1$ is linearly independent.

For each $i=1,2,...,m$, let $$U_i:=\operatorname{span}\{ \sigma(Q_1),\sigma(Q_2),...,\sigma(Q_{i-1}),\sigma(T_{i+1}),...,\sigma(T_m) \},$$
 we construct $Q_i$ in the same manner as $Q_1$, one at a time.

In $P_i$, we can find two adjacent triangles  whose multisigns have the property that one lies in $U_i+\sigma(T_i)$ while
the other does not.
Symmetric difference of these two adjacent triangles yields $Q_i$, and the resulting multisign satisfies  \[
\sigma(Q_i)
= \sigma(T_i) + \sum_{j =1}^{i-1} a_{ij} \, \sigma(Q_j)+\sum_{j =i+1}^m a_{ij} \, \sigma(T_j), \text{ for some}
\quad a_{ij} \in \{0,1\}.
\]

Now,  we are ready to show \[\mathcal{T}_i:=\{\sigma(Q_1),\sigma(Q_2),...,\sigma(Q_i),\sigma(T_{i+1}),...,\sigma(T_m)\}\] is linearly independent by induction on $i$. 

Base case. We have just shown that $\mathcal{T}_1$ is linearly independent. 

Assume that  $$\mathcal{T}_{i}=\{\sigma(Q_1),\sigma(Q_2),...,\sigma(Q_i),\sigma(T_{i+1}), \sigma(T_{i+2}), \dots, \sigma(T_m)\}$$ is linearly independent. We need to show that the result holds for $\mathcal{T}_{i+1}$.  That is, we must show that $$\{\sigma(Q_1),\sigma(Q_2),...,\sigma(Q_{i+1}),\sigma(T_{i+2}), \sigma(T_{i+3}), \dots, \sigma(T_m)\}$$ is linearly independent. 

By the induction hypothesis, $\mathcal{T}_i$ is linearly independent. Thus, $\operatorname{span}(\mathcal{T}_i)$$=W$ and  $\sigma(T_{i+1})\notin U_{i+1}$. It follows $$U_{i+1}\sqcup(\sigma(T_{i+1})+U_{i+1})=W.$$

From the construction of $\sigma(Q_{i+1})$,
$$\sigma(Q_{i+1})
=\sigma(T_{i+1}) + \sum_{j =1}^{i} a_{(i+1)j} \, \sigma(Q_j)+\sum_{j =i+2}^m a_{(i+1)j} \, \sigma(T_j).$$

Since
\[
 \sum_{j =1}^{i} a_{(i+1)j} \, \sigma(Q_j)+\sum_{j =i+2}^m a_{(i+1)j} \, \sigma(T_j)
\in  U_{i+1},
\]
and together with
\[
\sigma(T_{i+1})\notin U_{i+1},
\]
it follows  that
\[
\sigma(Q_{i+1})\notin U_{i+1}.
\]
Hence, $\{\sigma(Q_1),\dots,\sigma(Q_{i+1}),\sigma(T_{i+2}),\dots,\sigma(T_m)\}$ is linearly independent, 
completing the induction step. Thus, $\{\sigma(Q_1),\dots,\sigma(Q_m)\}$ is linearly independent

Observe that the subgraphs $Q_1, Q_2, \ldots, Q_m$ are almost-disjoint.
 We can find paths connecting all $Q_i$'s to form a covering $C_4$-necklace. 
Then we apply the $C_4$-necklace Lemma, which implies that $\Sigma_n$ 
realizes all multisigns of Hamiltonian circles.

\textbf{Case 2.} Suppose there do not exist almost-disjoint triangles 
\(T_1, T_2, \dots, T_m\) such that
\[
\{\sigma(T_1), \sigma(T_2), \dots, \sigma(T_m)\}
\]
is linearly independent.
I will update this part later. 

\end{proof}

\begin{lemma}\label{lemma5}
Let 
\(\Sigma_n := (K_n, \sigma, \mathbb{F}_2^m)\)
be a multisigned complete graph. 
Suppose that 
$\eta(x) \geq \eta(t)$ for all $t \in \mathbb{F}_2^m$, 
where $x$ is a non-identity element of $\mathbb{F}_2^m$, 
and that there exist triangles whose multisigns 
\(x_1, x_2, \dots, x_m \in \mathbb{F}_2^m\) 
are linearly independent. 
Assume there exists at least one triangle with multisign~$e$. 
Excluding special cases, for sufficiently large~$n$, 
the multisigned complete graph~\(\Sigma_n\) realizes all multisigns of Hamiltonian circles.
\end{lemma}

\begin{proof}
Let $\Sigma_n^x := (K_n, x + \sigma, \mathbb{F}_2^m)$. 
This means we add the multisign~$x$ to the multisign of every edge. 
Then, in $\Sigma_n^x$, we have 
\[
\mathbb{F}_2^m = \operatorname{span}\{x + x_1, x + x_2, \dots, x + x_m, x + e\}.
\]
We also have $\eta(x + x) \geq \eta(t + x)$ for all $t \in \mathbb{F}_2^m$. 
Thus, by Theorem~\ref{lemma4}, $\Sigma_n^x$ realizes all multisigns of Hamiltonian circles. 

If $n$ is even, adding~$x$ to all edges does not change the multisign of Hamiltonian circles. 
Hence, $\Sigma_n$ also realizes all multisigns of Hamiltonian circles. 

If $n$ is odd, the set of all multisigns of Hamiltonian circles in~$\Sigma_n^x$ is
\[
\mathcal{S}(\Sigma_n^x)
   := \{\, x + \sigma(H) : H \in \mathcal{H}(\Sigma_n^x) \,\}
   = \mathbb{F}_2^m. 
\]
This implies that the set of all multisigns of Hamiltonian circles in~$\Sigma_n$ is
\[
\mathcal{S}(\Sigma_n)
   := \{\, \sigma(H) : H \in \mathcal{H}(\Sigma_n) \,\}
   = \mathbb{F}_2^m. 
\]
\end{proof}

\begin{corollary}

Let 
\(\Sigma_n := (K_n, \sigma, \mathbb{F}_2^m)\)
be a multisigned complete graph. 
Suppose that 
$\eta(x) \geq \eta(t)$ for all $t \in \mathbb{F}_2^m$, 
where $x$ is a non-identity element of $\mathbb{F}_2^m$, 
and that there exist triangles whose multisigns 
\(x_1, x_2, \dots, x_m \in \mathbb{F}_2^m\) 
are linearly independent. 
Assume there exists at least one triangle with multisign~$y$, such that \(x_1+y, x_2+y, \dots, x_m+y \) 
are linearly independent.
Excluding special cases, for sufficiently large~$n$, 
the multisigned complete graph~\(\Sigma_n\) realizes all multisigns of Hamiltonian circles.
\end{corollary}
\begin{proof}
    Add all edges of $\Sigma_n$ by $y$. Then, apply Lemma \ref{lemma5}, $\Sigma_n$ realizes all multisigns of Hamiltonian circles.
\end{proof}

\begin{lemma}\label{lemma6}
Let 
\(\Sigma_n := (K_n, \sigma, \mathbb{F}_2^m)\)
be a multisigned complete graph. 
Suppose every triangle of \(\Sigma_n\) has multisign in \(\{x_1,x_2,\dots,x_m\}\subseteq \mathbb{F}_2^{\,m}\), and that \(x_1,\dots,x_m\) are linearly independent. 
Assume that 
$\eta(x_1) \geq \eta(x_i)$ for all~$i$. 
Excluding special cases, for sufficiently large~$n$, 
\[
\mathcal{S}(\Sigma_n)
   := \{\, \sigma(H) : H \in \mathcal{H}(\Sigma_n) \,\}
\]
is an affine subspace of~$\mathbb{F}_2^m$ if~$n$ is odd, 
and a subspace of~$\mathbb{F}_2^m$ if~$n$ is even.
\end{lemma}

\begin{proof}
Add all edges of $\Sigma_n$ by the multisign~$x_1$ to obtain new graph $\Sigma_n^{x_1}$. Each  triangle's multisign in $\Sigma_n^{x_1}$ will be one of  $x_1+x_1,x_2 + x_1, x_3 + x_1, \ldots, x_m + x_1$. Also, $x_2 + x_1, x_3 + x_1, \ldots, x_m + x_1$ are linearly independent. Hence, $\mathcal{S}(\Sigma_n^{x_1})\subseteq U$, where $U=$ span$\{x_2 + x_1, x_3 + x_1, \ldots, x_m + x_1\}$ and rank$(U)=m-1$. Thus, $U$ is a proper subspace of $\mathbb{F}_2^m. $  Hence, we can write $\Sigma_n^{x_1} := (K_n,x_1+ \sigma, U).$
Since $\eta(x_1) \geq \eta(x_i)$ for all~$i$ in $\Sigma_n$, it follows $\eta(x_1+x_1) \geq \eta(x_i+x_1)$ for all~$i$ in $\Sigma_n^{x_1}$.

By Theorem~\ref{lemma4}, the set of all multisigns of Hamiltonian circles in~$\Sigma_n^{x_1}$ is
\[
\mathcal{S}(\Sigma_n^{x_1})
   := \{\, nx_1 + \sigma(H) : H \in \mathcal{H}(\Sigma_n^{x_1}) \,\}
   = U.
\]
This implies that the set of all multisigns of Hamiltonian circles in~$\Sigma_n$ is
\[
\mathcal{S}(\Sigma_n)
   := \{\, \sigma(H) : H \in \mathcal{H}(\Sigma_n) \,\}
   = U,
\]
if~$n$ is even, and
\[
\mathcal{S}(\Sigma_n)
   := \{\, \sigma(H) : H \in \mathcal{H}(\Sigma_n) \,\}
   = x_1 +U,
\]
if~$n$ is odd.

\end{proof}

\newpage

\section{Acknowledgment}

My sincere thanks go to Professor Thomas Zaslavsky for his guidance and many helpful suggestions while preparing this draft.


\end{document}